%% file: main_s.tex
\RequirePackage{fix-cm}
\documentclass[smallextended]{svjour3}       % onecolumn (second format)
\smartqed  % flush right qed marks, e.g. at end of proof
\usepackage{graphicx}

\usepackage{amsmath}
\usepackage{mathrsfs, amssymb, amsthm, mathtools, bbm}

\usepackage[table, dvipsnames]{xcolor}
\usepackage{float} 

\usepackage[english]{babel} 
\usepackage[utf8]{inputenc} 
\usepackage[T1]{fontenc}

\usepackage[margin=1.2in,footskip=0.35in]{geometry} 

\usepackage[hidelinks]{hyperref}

\usepackage{tikz} 
\usetikzlibrary[topaths] 

\usepackage{footmisc} 

\usepackage{enumitem}

\usepackage{algorithm}
\usepackage{algcompatible} 

\usepackage{hyperref} 

\usepackage{makecell,booktabs}

\usepackage[square,numbers]{natbib} 

\input{macros_s} 

\begin{document}

\title{Revisit First-order Methods for Geodesically Convex Optimization
}


% \author{Yunlu Shu$^{a,}$\footnote{\textit{Email}: 22110840008@m.fudan.edu.cn; \textit{Address}: Songhu Rd. 2005, 200438, Yangpu District, Shanghai, China.} \quad Jiaxin Jiang$^{b,}$\footnote{\textit{Email}: jxjiang20@fudan.edu.cn; \textit{Address}: Handan Rd. 220, 200433, Yangpu District, Shanghai, China} \quad Lei Shi$^{b,c,}$\footnote{\textit{Email}:  leishi@fudan.edu.cn; \textit{Address}: Handan Rd. 220, 200433, Yangpu District, Shanghai, China} \quad Tianyu Wang$^{a,}$\footnote{Corresponding author. \textit{Email}: wangtianyu@fudan.edu.cn; \textit{Address}: Songhu Rd. 2005, 200438, Yangpu District, Shanghai, China.} \\ 
% \small $^{a}$Shanghai Center for Mathematical Sciences, Fudan University, Shanghai, China \\ 
% \small $^{b}$School of Mathematical Sciences, Fudan University, Shanghai, China \\ 
% \small $^{c}$Shanghai Key Laboratory for Contemporary Applied Mathematics, Fudan University, Shanghai, China
% } 

\author{Yunlu Shu $^{1}$             \and
        Jiaxin Jiang $^{2}$           \and
        Lei Shi $^{2,3}$      \and
        Tianyu Wang $^{1,\ast}$  
        % \authornote{*}   
}

% %\authorrunning{Short form of author list} % if too long for running head

\institute{Yunlu Shu \at
           ylshu22@m.fudan.edu.cn  
           \and
           Jiaxin Jiang \at
           jxjiang20@fudan.edu.cn 
           \and
            Lei Shi \at
            leishi@fudan.edu.cn 
           \and
           Tianyu Wang  \at
           wangtianyu@fudan.edu.cn
           \and
           $\ast$ Corresponding author
}

\footnotetext[1]{Shanghai Center for Mathematical Sciences, Fudan University, Shanghai, China}
\footnotetext[2]{School of Mathematical Sciences, Fudan University, Shanghai, China}
\footnotetext[3]{Shanghai Key Laboratory for Contemporary Applied Mathematics, Fudan University, Shanghai, China}
%\footnotetext[$\ast$]{Corresponding author}


\maketitle

\begin{abstract}

\input{tex/abs}

    % In a seminal work of Zhang and Sra, gradient descent methods for geodesically convex optimization were comprehensively studied. In particular, Zhang and Sra derived a comparison inequality that relates the iterative points in the optimization process. Since their seminal work, numerous follow-ups have studied different downstream usages of their comparison lemma. However, all results along this line relies on strong assumptions, such as bounded domain assumption or curvature bounded below assumption. 

    % In this work, we introduce the concept of quasilinearization to optimization, presenting a novel framework for analyzing geodesically convex optimization. By leveraging this technique, we establish state-of-the-art convergence rates -- for both deterministic and stochastic settings -- under substantially weaker assumptions than previously required. 
    
    %\textbf{MSC codes:} 90C25, 90C15
    
\end{abstract}

\input{tex/intro_s}

\input{tex/reorganized-2}

\section{Conclusion}

In this work, we propose proximal-based methods for geodesically convex optimization on Hadamard manifolds. We focus on optimizing $g$-convex functions in both deterministic and stochastic settings. By leveraging the quasilinearized inner product, we eliminate strongly assumptions required by previous works, and rigorously establish convergence rates for these optimization problems on Hadamard manifolds.

By unifying generality with efficiency, our framework resolves the tension between restrictive geometric assumptions and practical convergence guarantees, advancing the applicability of convex optimization in non-Euclidean spaces.

On a broader scale, the concept of quasilinearization can be extended to analyze a wide variety of optimization algorithms. This includes sub-gradient methods for nonsmooth optimization problems, stochastic variance-reduced gradient (SVRG) techniques, and accelerated optimization frameworks derived from quasilinearization principles, among others. The versatility of quasilinearization allows it to provide valuable insights into the convergence behavior and efficiency of these methods, making it a powerful tool in both theoretical and applied optimization research.

\begin{acknowledgements}
The authors thank Nicolas Boumal for insightful comments. 
\end{acknowledgements}


\bibliographystyle{spmpsci}
\bibliography{references} 

\appendix
\input{./tex/appendix}

\end{document}

%% file: macros_s.tex
\newtheorem{assumption}{Assumption}

\newtheorem*{theorem*}{Theorem}
\newtheorem*{example*}{Example} 
\newtheorem*{definition*}{Definition}
\newtheorem*{lemma*}{Lemma}
\newtheorem*{remark*}{Remark}
\newtheorem*{corollary*}{Corollary}
\newtheorem*{proposition*}{Proposition}
\newtheorem*{assumption*}{Assumption}
\newtheorem*{claim*}{Claim}

\newtheoremstyle{TheoremNum}
        {\topsep}{\topsep}              %%% space between body and thm
        {\itshape}                      %%% Thm body font
        {}                              %%% Indent amount (empty = no indent)
        {\bfseries}                     %%% Thm head font
        {.}                             %%% Punctuation after thm head
        { }                             %%% Space after thm head
        {\thmname{#1}\thmnote{ \bfseries #3}}%%% Thm head spec
\theoremstyle{TheoremNum}

\newtheoremstyle{LemmaNum}
        {\topsep}{\topsep}              %%% space between body and thm
        {\itshape}                      %%% Thm body font
        {}                              %%% Indent amount (empty = no indent)
        {\bfseries}                     %%% Thm head font
        {.}                             %%% Punctuation after thm head
        { }                             %%% Space after thm head
        {\thmname{#1}\thmnote{ \bfseries #3}}%%% Thm head spec
\theoremstyle{LemmaNum}

\newcommand{\x}{ \mathbf{x} }
\newcommand{\y}{ \mathbf{y} }
\newcommand{\z}{ \mathbf{z} }
\renewcommand{\v}{ \mathbf{v} }
\renewcommand{\u}{ \mathbf{u} }
\newcommand{\w}{ \mathbf{w} }

\newcommand{\oa}[1]{ \overrightarrow{#1} }

\newcommand{\cosq}{ \mathrm{cosq} }

\newcommand{\grad}{\mathrm{grad}}

\newcommand{\M}{\mathcal{M}}
\newcommand{\D}{\mathcal{D}}
\newcommand{\Exp}{{\mathrm{Exp}}}

\newcommand{\E}{\mathbb{E}}

\renewcommand{\[}{\left[ }
\renewcommand{\]}{\right] }

\newcommand{\<}{\left< }
\renewcommand{\>}{\right> }

\renewcommand{\(}{\left( }
\renewcommand{\)}{\right) }

%% file: tex/abs.tex
In a seminal work of Zhang and Sra, gradient descent methods for geodesically convex optimization were comprehensively studied. In particular, Zhang and Sra derived a comparison inequality that relates the iterative points in the optimization process. Since their seminal work, numerous follow-ups have studied different downstream usages of their comparison lemma. 
% However, all results along this line relies on strong assumptions, such as bounded domain assumption or curvature bounded below assumption. 

In this work, we introduce the concept of quasilinearization to optimization, presenting a novel framework for analyzing geodesically convex optimization. By leveraging this technique, we establish state-of-the-art convergence rates -- for both deterministic and stochastic settings -- under weaker assumptions than previously required. The technique of quasilinearization may prove valuable for other non-Euclidean optimization problems. 
% The language of quasilinearization might also be of interest for other optimization problems in non-Euclidean spaces. 

%% file: tex/intro_s.tex
\section{Introduction}

% Geodesically convex optimization is an emerging area of study that merges concepts from differential geometry with optimization theory, creating a rich framework for tackling complex problems across various fields, including machine learning, control theory, and data science. 
Geodesically convex optimization integrates concepts from differential geometry with optimization theory, creating a robust framework for addressing complex problems across various fields, including machine learning, economics, data science, and numerical PDEs. 
%Unlike traditional convex optimization, which relies on the linear structures of underlying spaces, geodesically convex optimization focuses on spaces equipped with a Riemannian metric. 
It generalizes convex optimization by focusing on Riemannian metrics rather than linear structures.
This field allows for the exploration of optimization landscapes that are inherently curved, moving beyond the linear structures typically characterized by Euclidean, Hilbert, and Banach spaces.

A function is geodesically convex ($g$-convex) if its behavior mirrors that of standard convex functions along geodesics. This generalization enables the optimization of functions that may not exhibit traditional convexity but still possess desirable properties along specific paths. As a result, geodesically convex optimization opens up new avenues for solving problems in non-Euclidean spaces, where traditional methods may falter.

To this end, we study optimization problems: 
\begin{align}
    \min_{x \in \mathcal{M}} f (x) 
\end{align}
where $ \M $ is a Hadamard manifold endowed with a Riemannian metric $g$. Over the past few years, many researchers have contributed to this study of this problem \citep{absil2008optimization,bacak2014convex,zhang2016first,bergmann2016parallel,lerkchaiyaphum2017iterative,bredies2020first,khan2020multistep,criscitiello2022negative,hirai2023convex,sakai2023convergence,hirai2024gradient}. In the seminal works of Bonnabel \citep{bonnabel2013stochastic} and  Zhang and Sra \citep{zhang2016first}, they build up an analysis framework that leverages the geodesic triangle defined by the current iterate $x_t$, the subsequent iterate $x_{t+1}$, and the optimal point $x^*$. 
% In the study of iterative algorithms on Riemannian manifolds, a crucial step involves understanding the geodesic triangle defined by the current iterate $x_t$, the subsequent iterate $x_{t+1}$, and the optimal point $x^*$. For gradient descent applied to geodesically convex functions, Zhang and Sra utilized the triangle comparison theorem to derive a result that establishes a relationship between these points. 
Specifically, they utilized the triangle comparison theorem of Toponogov and ingeniously proved the following proposition. 

\begin{proposition}[Corollary 8 in Zhang and Sra \citep{zhang2016first}]
    \label{thm:gd-compare}
    For any Riemannian manifold $\mathcal{M}$ where the sectional curvature is lower bounded by $\kappa$ and any point $x, x_s \in \M$ (such that $d(x,x_s)$ is less than the injectivity radius of $x_s$), the update $x_{s+1} = \Exp_{x_s}(-\eta_s \grad f(x_s))$ satisfies
    \begin{align} 
        \langle - \grad f(x_s), \mathrm{Exp}^{-1}_{x_s}(x) \rangle_{x_s} \leq \frac{1}{2\eta_s} \left( d^2(x_s,x) - d^2(x_{s+1},x) \right) + \frac{\zeta(\kappa,d(x_s,x))\eta_s}{2} \| \grad f(x_s) \|^2, \label{eq:critical}
    \end{align} 
    where $\zeta (\kappa, c) := \frac{\sqrt{|\kappa|}c}{\tanh(\sqrt{|\kappa|}c)} $. 
\end{proposition} 

% In the work of Zhang and Sra and subsequent works, Theorem \ref{thm:gd-compare} is extensively used in the analysis of first-order methods for geodesically convex optimization. This crucial inequality (\ref{eq:critical}) has two advantages: 
% \begin{itemize} 
%     \item The left-hand-side of (\ref{eq:critical}) directly relates to geodesic convexity. 
%     \item The right-hand-side of (\ref{eq:critical}) can be telescoped. 
% \end{itemize} 

In the seminal works of Bonnabel \citep{bonnabel2013stochastic} and  Zhang and Sra \citep{zhang2016first}, as well as many subsequent studies \citep[e.g.,][]{zhang2016riemannian,weber2017frank,zhang2018estimate,tripuraneni2018averaging,sun2019escaping,lin2020accelerated,kim2022accelerated,alimisis2021momentum,kim2022accelerated,martinez2022global,sakai2023convergence,martinez2024convergence}, Proposition \ref{thm:gd-compare} and similar results have been extensively utilized in the analysis of first-order methods for geodesically convex optimization. This proposition is particularly significant due to following advantages: 

\begin{itemize} 
    \item The left-hand side of (\ref{eq:critical}) is inherently linked to the concept of geodesic convexity ($g$-convexity). 
    %This connection is crucial as it allows the inequality to align seamlessly with the geometric properties of the optimization problem on Riemannian manifolds, establishing a clear relationship between the algorithm's progress and the underlying convexity structure.
    This connection is crucial as it directly associates the algorithm's progress with the underlying convexity structure.
    While the right-hand side of (\ref{eq:critical}) possesses the property of telescoping, which is highly beneficial in the analysis of iterative algorithms. 
    %Telescoping facilitates the summation of inequalities across multiple iterations, leading to a straightforward derivation of convergence rates. This property streamlines the analysis and fosters a more intuitive understanding of how the algorithm advances toward the optimal solution.
    %Telescoping facilitates the summation over iterations and greatly simplifies the derivation of convergence rates.
\end{itemize} 

Nonetheless, (\ref{eq:critical}) also suffers limitations:
\begin{itemize}
    \item The right-hand side of (\ref{eq:critical}) incorporates the term $\zeta(\kappa,d(x_s,x)) $, which requires a lower bound on the sectional curvature as well as an upper bound on $ d(x_s,x) $. This requirement imposes two assumptions: \textbf{(A1)} a uniform lower bound on the sectional curvature, and \textbf{(A2)} an \emph{a priori} upper bound on the diameter of the trajectory $ \{ x_s \}_s $. 
    % \emph{a-priori}. 
\end{itemize} 

In particular, this Curvature Bounded Below (CBB) condition \textbf{(A1)} fails to hold for a large range of Hadamard manifolds. For example, warped product manifold and doubly warped product manifold can construct a large class of Hadamard manifolds with curvature tends to negative infinity \citep{bishop1969manifolds,petersen2006riemannian}.

\begin{example}
    Starting with a Riemannian manifold $(F, h)$, we consider a warped product manifold:
    \begin{align*}
        \M = \mathbf{I} \times_{\phi} F; \text{ with metric }
        g =dr^2+\phi(t)^2 h,
    \end{align*}
    where $\mathbf{I} \subset \mathbb{R}$ is an open interval and $\phi$ is a positive, differentiable function on $\mathbf{I}$. 
    %The warped product $\M$ is the manifold $\mathbf{I} \times F$ endowed with the metric
    % \begin{align*}
    %     g =dr^2+\phi(t)^2 h.
    % \end{align*}

    By Bishop and O’Neill \citep{bishop1969manifolds} (Lemma 7.5), the warped product $\M = \mathbf{I} \times_{\phi} F$ has non-positive curvature ($K \le 0$) if $\phi$ is convex and $dim(F)=1$. 
    %In this case, we make the additional assumption that $F$ has sectional curvature $0$. (error, sectional curvature no definition)
    In this case, the sectional curvature for a tangent plane $\Pi$ at $(t; p)$ is
    %In such cases, for a tangent plane $\Pi$ at $(t; p)$, the sectional curvature formula shows
    \begin{align*}
        K(\Pi) = -\frac{\phi^{\prime \prime}(t)}{\phi(t)}\|x\|^2+\frac{-{\phi^{\prime}}^2(t)}{\phi^2(t)}\|v\|^2,
    \end{align*}
     where $\|x\|^2+\|v\|^2=1$ ($x$ horizontal and $v$ vertical).

    Specifically, one can take $I = (0,1)$ and $ \phi (t) = t^2 $ to get a quick example manifold whose curvature lower bound goes to negative infinity. 
    % Consequently, one can easily construct suitable functions $\phi$ for which the curvature lower bound tends to negative infinity.
    \label{ex:boudly}
\end{example}

This raises a crucial question:

\begin{quote}
    \textit{Is it possible to remove both Assumption \textbf{(A1)} and Assumption \textbf{(A2)} while still maintaining the state-of-the-art convergence rate?} \textbf{(Q)}
\end{quote}

% Over the past decade, despite considerable efforts to tackle Question \textbf{(Q)} from various perspectives \citep[e.g.,][]{NIPS2017_6ef80bb2,tripuraneni2018averaging,10.5555/3692070.3693490}, no prior work has succeeded in resolving this question. 
Although Question \textbf{(Q)} has attracted significant attention over the past decade \citep[e.g.,][]{liu2017accelerated,tripuraneni2018averaging,zhang2018estimate,martinez2024convergence}, all attempts to resolve it have thus far fallen short. 
For instance, the work by Martínez-Rubio et al. \citep{martinez2024convergence} analyzes the proximal methods over Riemannian manifolds, and removes the requirement for Assumption \textbf{(A2)}, yet their analysis remains within the framework of Proposition \ref{thm:gd-compare} and still requires Assumption \textbf{(A1)}. At this point, it is evident that resolving Question \textbf{(Q)} poses a substantial challenge -- likely demanding a fundamentally new methodology that cannot rely on Proposition \ref{thm:gd-compare} or similar techniques. In this work, we provide an affirmative answer to Question \textbf{(Q)}, by establishing new convergence guarantees for optimization over Hadamard manifolds. 

Our work makes the following primary contributions, spanning both deterministic and stochastic $g$-convex optimization; Please see Table \ref{tab:compare-deter} for a detailed comparison to prior arts.
For deterministic optimization problems
%~\footnote[1]{The smoothness requirement can be relaxed; See Remark \ref{remark:relax} for details.}
: 
% ~\arabic{footnote}{ }:
\begin{itemize} 
    % \item If the objective function is strongly $g$-convex (geodesically strongly convex) and smooth, then the gradient descent algorithm achieves linear convergence, without requiring \textbf{(A1)} or \textbf{(A2)} or their alternatives. 
    % without any assumptions regarding a lower bound on curvature or a bounded domain. 
    \item If the objective function is $g$-convex (geodesically convex), then a proximal gradient algorithm achieves a $ \mathcal{O} (1/t) $ convergence rate, without requiring \textbf{(A1)} or \textbf{(A2)}. 
    % without any assumptions on curvature lower bounds or bounded domains. 
\end{itemize}

For stochastic optimization:
\begin{itemize} 
    % \item If the overall objective is strongly $g$-convex (geodesically strongly convex) and smooth, then the stochastic gradient descent algorithm achieves an $ O(1/t) $ convergence rate, without requiring \textbf{(A1)} or \textbf{(A2)} or their alternatives. 
    % without any assumptions regarding curvature lower bounds or bounded domains. 
    \item If the objective function is $g$-convex (geodesically convex) and smooth, then a stochastic proximal gradient algorithm achieves a $ \widetilde{O}(1/\sqrt{t}) $ convergence rate, without requiring \textbf{(A1)} or \textbf{(A2)}.  
    % without any assumptions on curvature lower bounds or bounded domains. 
\end{itemize} 

% Our results are proved via new techniques that may also be useful for more Riemannian optimization problems. 

Our results are established through new techniques that may be useful for a wider range of Riemannian optimization problems. Below, Tables \ref{tab:compare-deter} provides a comprehensive comparison of our work with state-of-the-art methods.
It shows the assumptions and convergence behaviors of existing approaches in the deterministic case. 
% Table \ref{tab:compare-stoc} presents analogous results in stochastic case.
% These tables reveal that our method eliminates the demand for curvature lower bound and bounded domain entirely. This allows our framework to operate on arbitrary Hadamard manifolds.
As shown in the table, our algorithm attains state-of-the-art convergence rates for $g$-convex optimization while relying on weaker assumptions. 

\begin{table}[h!] 
    \centering 
    \begin{tabular}{c|c|c|c|c} 
    \toprule \hline 
        % \multicolumn{5}{c}{\makecell{ \textbf{Strongly $g$-convex objectives, in deterministic environments}}} \\ \hline 
        % & \makecell{Need Curvature\\Lower Bound?} & \makecell{Need Bounded \\Domain?} & \makecell{Need Solving\\Eqn. Involving\\$\Exp^{-1}$ \& $\Gamma $?} & \makecell{Convergence\\Rate} \\ \hline 
    % \citet{zhang2016first} & Yes & Yes & No & Linear \\ \hline 
    % \rowcolor{lightgray}\citet{liu2017accelerated} & No & No & Yes & Linear${}^\dagger$ \\ \hline 
    % \citet{tripuraneni2018averaging} & ----- & ----- & ----- & ----- \\ \hline 
    % \rowcolor{lightgray} \citet{zhang2018estimate} & Yes & Yes & No & Linear \\ \hline 
    % \citet{becigneul2018riemannian} & ----- & ----- & ----- & ------ \\ \hline 
    % \citet{ferreira2019gradient} & ----- & ----- & ----- & ------ \\ \hline 
    % \rowcolor{lightgray} \citet{kim2022accelerated} & Yes & Yes & No & Linear \\ \hline 
    % \rowcolor{lightgray}
    % \citet{jin2022understanding} & Yes & No & No & Linear \\ \hline 
    % \citet{martinez2024convergence} & Yes & No & No & Linear \\ \hline 
    % \textbf{This Work (Thm. \ref{thm:strong-conv})}  & \textbf{No} & \textbf{No} & \textbf{No} & \textbf{Linear} \\ 
    % \toprule \hline 
        \multicolumn{5}{c}{\makecell{ \textbf{$g$-convex objectives, in deterministic environments}}} \\ \hline 
        & \makecell{Need Curvature\\Lower Bound?} & \makecell{Need Bounded \\Domain a-priori?} & \makecell{Need Solving\\Eqn. Involving\\$\Exp^{-1}$ \& $\Gamma $?} & \makecell{Convergence\\Rate} \\ \hline 
    Zhang and Sra \citep{zhang2016first} & Yes & Yes & No & $\mathcal{O}(t^{-1})$ \\ \hline 
    \rowcolor{lightgray} Liu et al. \citep{liu2017accelerated} & No & Yes & Yes & $\mathcal{O}^\dagger (t^{-2})$ \\ \hline 
    %\citet{tripuraneni2018averaging} & ----- & ----- & ----- & ----- \\ \hline 
    %\rowcolor{lightgray}\citet{zhang2018estimate} & ----- & ----- & ----- & ----- \\ \hline 
    %\citet{becigneul2018riemannian} & ----- & ----- & ----- & ------ \\ \hline 
    Ferreira et al. \citep{ferreira2019gradient} & Yes & No & No & $\mathcal{O} (t^{-1})$ \\ \hline 
    \rowcolor{lightgray} Kim and Yang \citep{kim2022accelerated} & Yes & Yes & No & $\mathcal{O} (t^{-2})$ \\ \hline 
    % \rowcolor{lightgray}
    % %\citet{jin2022understanding} & ----- & ----- & ----- & ----- \\ \hline 
    % \makecell{Martínez-Rubio et al.\\ \citep[RIPPA in][]{martinez2024convergence}} & Yes & No & No & $\mathcal{O} (t^{-1})$ \\ \hline 
    \makecell{Martínez-Rubio et al. \citep{martinez2024convergence}} & Yes & No & No & $\mathcal{O} (t^{-1})$ \\ \hline 
    \textbf{This Work (Thm. \ref{thm:prox})}  & \textbf{No} & \textbf{No} & \textbf{No} & $\mathcal{O} (t^{-1} )$ \\ \hline \bottomrule
    \end{tabular} 
    \caption{Comparison to state-of-the-art results for {first-order methods} on $g$-convex optimization problems over Hadamard manifolds, in deterministic environments. 
    The light gray rows in this table highlight acceleration methods, which are outside the primary focus of the this work. 
    % In this table, the rows shaded in lightgray marks acceleration methods, which is not the focus of this work. 
    % In this table, the convergence rates annotated with ${}^\dagger$ indicates that solving an additional equation in each iteration is required, which may incur extra computational complexity; 
    % In this table, the convergence rates marked with ${}^\dagger$ indicate that solving an additional equation per iteration is required, potentially incurring additional computational complexity; 
    The convergence rates marked with ${}^\dagger$ indicate that it is necessary to solve an extra equation in each iteration, which may incur extra computational complexity; Such requirement is considered computationally intractable by some authors \citep{kim2022accelerated}. 
    %Cells marked with ``-----'' indicates that no results available. 
    } 
    \label{tab:compare-deter}
\end{table}

\subsection{Prior Arts}
Convex optimization techniques over manifolds have been a central topic in contemporary optimization. This focus arises from that many optimization problems are posed on manifolds, such as geometric models for the human spine \citep{adler2002newton}, eigenvalue optimization problems \citep{absil2008optimization}, and so on.

This development drives the need to expand algorithms from Euclidean spaces to Riemannian manifolds. Indeed, some important methodologies on linear spaces such as gradient descent, Newton’s method, the proximal point method, and so on have been adapted to Riemannian manifolds 
%subdifferentials, the conjugate gradient method, 
\citep{bonnabel2013stochastic,newton2018stochastic,adler2002newton,azagra2005nonsmooth,bento2012subgradient,dedieu2003newton,ledyaev2007nonsmooth,li2009monotone,li2009existence,nemeth2003variational,smith1994optimization,udriste1994convex,liu2017accelerated,ferreira2019gradient,jin2022understanding}. Nevertheless, numerous algorithms remain worthy of deeper study. 
% Yet many algorithms still merit deeper investigation. 

% only offer adaptations to Riemannian manifolds via straightforward analysis. The development of these methods on Hadamard manifolds merit deeper and more comprehensive investigation. 

Hadamard manifold -- a (simply connected) Hadamard space (complete CAT(0) spaces) equipped with a Riemannian metric -- is an important class of nonlinear Riemannian manifolds \citep{bacak2014convex,ballmann1995lectures,ballmann1985manifolds,bridson1999metric}. 
% It comprises Hilbert spaces, nonlinear Lebesgue spaces and many other spaces. 
The properties of Hadamard manifolds have been extensively studied and have long been a focal point of interest in geometric analysis \citep{bavcak2023old}.
The concept of Hadamard spaces originated with Wald \citep{wald1936begrundung} and was advanced by Aleksandrov \citep{aleksandrov1951theorem}, who introduced the name ``Aleksandrov spaces of nonpositive curvature''. Later, Gromov introduced the terminology CAT(0).
% The concept of Hadamard spaces originated in a 1936 paper by \cite{wald1936begrundung}. Its significance gained prominence through Aleksandrov's pivotal contributions in the 1950s \citep{aleksandrov1951theorem}, leading to the designation ``Aleksandrov spaces of nonpositive curvature''. Gromov later introduced the terminology CAT(0). Since then, Hadamard spaces have been alternatively called complete CAT(0) spaces.
In 2008, Berg and Nikolaev \citep{Berg2008} investigated the curvature properties of Aleksandrov spaces via quasilinearization techniques.

% Convex optimization theory in Hadamard manifolds is a promising research frontier, as problems in various fields have been formulated via geodesically convex optimization on Hadamard spaces \citep{absil2008optimization,agueh2011barycenters,hirai2023convex}. 
% Examples include phylogenetic tree analysis and submodular minimization on modular lattices \citep{billera2001geometry,hamada2017maximum}.

%For instance, \cite{billera2001geometry} endowed biological phylogenetic trees with CAT(0) cubical complex structures. Also, \cite{hamada2017maximum} demonstrated the efficacy of Hadamard space methodologies for minimizing submodular functions over modular lattices. 
%Convex optimization theory gains further impetus from the following phenomenon: inherently non-convex optimization problems may become convex when reformulated through a proper metric \citep{absil2008optimization,agueh2011barycenters}. 

The analysis of convex function optimization over Hadamard manifolds plays a pivotal role in computational geometric analysis \citep{jost1995convex,jost1997nonpositive}. Some efforts have focused on extending results from the Euclidean spaces to Hadamard manifolds \citep{wang2011modified,ardila2014moving,bacak2014convex,bacak2014new,bacak2015convergence,bacak2016lipschitz,banert2014backward,bergmann2016parallel,lerkchaiyaphum2017iterative,huang2019riemannian,bredies2020first,bergmann2024difference}. For example, building on proximal point algorithm, Khan and Cholamjiak \citep{khan2020multistep} proposed a multi-step approximant.
%for convex optimization problem in Hadamard spaces. 
Hamilton and Moitra \citep{hamilton2021no}; Criscitiello and Boumal \citep{criscitiello2022negative} showed the impossibility to accelerate any deterministic first-order algorithm for a large class of Hadamard manifolds. 

Recent years have witnessed substantial advances in analyzing gradient descent algorithms on Hadamard manifolds \citep[e.g.,][]{zhang2016riemannian,weber2017frank,zhang2018estimate,tripuraneni2018averaging,sun2019escaping,lin2020accelerated,kim2022accelerated,alimisis2021momentum,kim2022accelerated,martinez2022global,sakai2023convergence,martinez2024convergence,sakai2023convergence,hirai2024gradient}. Among these, Zhang and Sra \citep{zhang2016first} offered the most relevant study of first-order optimization in this setting. However, the existing works rely on restrictive assumptions about the functions or manifolds, or limits its scope to specialized subproblems -- highlighting the need for more general analysis frameworks. 

This paper is organized as follows. Section \ref{section:preliminary} introduces the necessary concepts and quasilinearization. Sections \ref{section:method} is devoted to the algorithms and results in the deterministic and stochastic cases. Additional discussions on proximal operator is included in Section \ref{section:proximal} .

%% file: tex/reorganized-2.tex
\section{Quasilinearization of Hadamard Manifolds}
\label{section:preliminary}
This section introduces the necessary notation for analyzing manifolds and functions defined on them. We also formalize the concept of quasilinearization \citep{Berg2008}, a key tool for studying geometric properties of Hadamard manifolds.

\subsection{Hadamard manifolds and Notations}
% In this paper, we focus on Hadamard manifolds (or Cartan--Hadamard manifolds). 
Hadamard manifolds (or Cartan--Hadamard manifolds) are complete and simply connected Riemannian manifolds with everywhere non-positive sectional curvature \citep{sakai1996riemannian,lee2006riemannian}. Hadamard manifolds encompass not only Euclidean spaces but also more complex geometries, such as spaces with constant negative curvature (e.g., hyperbolic spaces) and those with variable non-positive curvature (e.g., the space of symmetric positive definite matrices). 

% All Riemannian manifold considered in this paper will be Hadamard manifold and will be denoted as $\M$.

A Hadamard manifold possesses several key properties about its geometric structure. First, by the Hopf--Rinow theorem, it is geodesically complete, meaning that every geodesic can be extended indefinitely. Futhermore, for any two points in the Hadamard manifold, there exists a unique geodesic connecting them. Additionally, the Cartan--Hadamard theorem ensures that Hadamard manifolds are diffeomorphic to some Euclidean space, and the exponential map at any point is bijective.

Let $( \M , g )$ be an Hadamard manifold endowed with a Riemannian metric $g$, and let $d$ be the distance metric over $\M$ associated with $g$. 
We use the following notations. 
For any $\x \in \M$, $T_{\x}\M$ denotes the tangent space to $\M$ at $\x$. Riemannian metric $g$ derives metric $g_{\x}$ on $T_{\x}\M$, which is compatible with $g$. We use $ \< \cdot, \cdot \>_\x $ to denote the inner product, and $ \| \cdot \|_\x $ to denote the norm on $T_{\x}\M$. The subscript of $\< \cdot, \cdot \>_\x$ cannot be ignored, as $ \< \cdot ,\cdot \> $ denotes the quasilinearized inner product (defined in Section \ref{sec:quasi}). Obviously, $(T_{\x}\M, g_{\x})$ is a Riemannian manifold that has constant sectional curvature 0.

Also, for any $\x \in \M$ and $\v \in T_\x \M$, there exists a unique geodesic $\gamma_{\v} $ through $\x$ whose tangent at $\x$ is $\v$. The exponential map $\Exp_\x : T_{\x}\M \rightarrow \M$ is defined by $ \Exp_\x (\tau \v) = \gamma_{\v} (\tau) $ ($ \tau \in [0,1] $) for all $\v \in T_{\x}\M$.
For $ \x, \y \in \M $, we use $ \oa{\x\y} $ to denote the ordered shortest geodesic segment from $\x$ to $\y$. We refer to $\x$ (resp. $\y$) as the start point (resp. end point) of the geodesic segment $\oa{\x\y}$, and use $ |\oa{\x\y}| := d (\x, \y) $ to denote the length of $ \oa{\x\y}$.

For any $\x, \y \in \M$, $\Gamma_{\y}^{\x}$ denotes the parallel transport from $T_\y \M$ to $T_\x \M$ along the minimizing geodesic. For any $\v \in T_\x \M$ and $\w \in T_\y \M$, parallel transport operation can “transport” $\v$ to $\w$. Then the norm of difference between two vectors from two tangent spaces is $\| \v - \Gamma_{\y}^{\x} \w \|_{\x}$.

%\begin{remark}
%    If $\M$ is a Riemannian manifolds, each tangent manifold $T_\x \M$ ($\x \in \M$) is naturally endowed with an inner product, which is compatible with the Riemannian metric $g$. We use $ \< \cdot, \cdot \>_\x $ to denote the inner product in $ T_\x \M $, and use $ \| \cdot \|_\x $ to denote the norm in this manifold. It is worth noting that this inner product $ \< \cdot, \cdot \>_\x $ is different from the quasilinearized inner product $ \< \cdot ,\cdot \> $. 
%\end{remark} 

\begin{remark}
    Throughout the rest of the paper, we use $g$ to denote the Riemannian metric, and use $d$ to denote to distance metric induced by $g$. Sometimes $d$ is simply referred to as metric. 
    % metric induced by the Riemannian length. 
\end{remark}

\begin{remark}
    When the expression is unambiguous, sometimes we use abbreviated notation for simplicity: We may omit the subscript $\x$ in the norm --- writing $\| \cdot \| := \| \cdot \|_{\x}$ when there is no confusion. 
    % And we typically omit the notation for parallel transport, meaning $\| \v - \w \|$ can represent $\| \v - \Gamma_{\y}^{\x} \w \|_{\x}$, when $\v \in T_\x \M$, $\w \in T_\y \M$.
\end{remark}

\subsection{Quasilinearization}
\label{sec:quasi}

In answering a question of Gromov \citep{gromov1999metric}, Berg and Nikolaev \citep{Berg2008} invented the notion of quasilinearization, which generalizes inner products from Euclidean manifolds to Hadamard manifolds. Now we introduce the notion of quasilinearization to optimization problems. 

Quasilinearization plays a pivotal role in our work. 
Existing techniques, such as the triangle comparison trick employed by Zhang and Sra \citep{zhang2016first}, require a curvature lower bound. In contrast, quasilinearization properties hold universally across all Hadamard manifolds, regardless of whether a curvature lower bound is assumed. 
% While our algorithm operates on any Hadamard manifold (even those with unbounded curvature), 
% For previous the lack of curvature constraints prevents the use of generalized cosine formula or related triangle inequalities. 
% While existing techniques, such as the triangle comparison trick used by \citet{zhang2016first}, requires a curvature lower bound, 
% % Nevertheless, 
% the properties of quasilinearization remain universally valid across all Hadamard manifolds, regardless of whether a curvature lower bound is assumed.

\begin{definition}[quasilinearized inner product]
    \label{def:quail-inner-prod}
    For any two ordered geodesic segments $\oa{\x\y}$ and $ \oa{\z\w} $ on manifold $\mathcal{M}$, the quasilinearized inner product is defined as 
\begin{align} 
    \< \oa{\x\y}, \oa{\z\w} \> 
    =
    |\oa{\x\y}| |\oa{\z\w} | \cosq \( \oa{\x\y}, \oa{\z\w} \) , \label{eq:def-inner}
\end{align}
where $ |\oa{\x\y}| $ is the length of $ \oa{\x\y} $, and 
\begin{align*}
    \cosq \( \oa{\x\y}, \oa{\z\w} \) 
    = 
    \frac{ | \oa{\x\w} |^2 + | \oa{\y \z} |^2 - | \oa{\x \z} |^2 - | \oa{\y \w} |^2 }{ 2 | \oa{\x\y} | | \oa{\z\w} | } . 
\end{align*} 
\end{definition}

This definition of the quasilinearized inner product can be applied to any two geodesic segments, even if they don't share a same end point. This quasilinearized inner product's magnitude is determined by the distances between the four points $\x, \y, \z$ and $\w$,
\begin{align*}
    \< \oa{\x\y}, \oa{\z\w} \> 
    =
    \frac{ | \oa{\x\w} |^2 + | \oa{\y \z} |^2 - | \oa{\x \z} |^2 - | \oa{\y \w} |^2 }{ 2 } . 
\end{align*}
In particular, if two of the points overlap, the expression simplifies to:
\begin{align*}
    \< \oa{\x\y}, \oa{\x\w} \> 
    =
    \frac{ | \oa{\x \y} |^2 + | \oa{\x\w} |^2 - | \oa{\y \w} |^2 }{ 2 } . 
\end{align*}

An intriguing attribute of this quasilinearized inner product is that it is compatible with an ``addition'' rule, which we describe now. For any three points $ \x, \y, \z $, we define $ \oa{\x\y} + \oa{\y\z} = \oa{\x\z} $. With this notion of addition, the quasilinearied inner product (\ref{eq:def-inner}) satisfies 
\begin{align*} 
    \< \oa{\x\y} + \oa{\y\z} , \oa{\u\w} \> = \< \oa{\x\y} , \oa{\u\w} \> + \< \oa{\y\z} , \oa{\u\w} \> . 
\end{align*} 

\begin{remark}
    \label{rem:add}
    The addition operation ``+'' between two geodesic line segments (e.g., $ \oa{\x\y} + \oa{\y\z}$) is well-defined only when the end point of the first operand coincides with the start point of the second operand. We say two geodesic line segments are addable when the addition operation ``+'' between them is well-defined. 
\end{remark} 

This quasilinearized inner product satisfies some additional properties, which are listed below in Proposition \ref{prop:qip}. 
% is that it follows the linear law whenever two ordered geodesics are 

\begin{proposition}[Berg and Nikolaev  \citep{Berg2008}] 
    \label{prop:qip}
    A quasilinearized inner product (\ref{eq:def-inner}) on a Hadamard space $\M $ with distance metric $d$ satisfies the following conditions:
    \begin{itemize}
        \item[\textbullet] $\< \oa{\x\y} , \oa{\x\y} \> = d^2 (\x,\y) , \quad \forall \x,\y \in \M $; 
        \item[\textbullet] $\< \oa{\x\y} , \oa{\z\w} \> = \<  \oa{\z\w} , \oa{\x\y} \> , \quad \forall \x,\y, \z, \w \in \M $; 
        \item[\textbullet] $\< \oa{\x\y} , \oa{\z\w} \> = - \< \oa{\y\x} , \oa{\z\w} \> , \quad \forall \x,\y, \z, \w \in \M $; 
        \item[\textbullet] $\< \oa{\x\y} + \oa{\y\z} , \oa{\u\w} \> = \< \oa{\x\y} , \oa{\u\w} \> + \< \oa{\y\z} , \oa{\u\w} \> $ for any $ \x,\y,\x,\u, \w \in \M $, where the addition operation is defined as $ \oa{\x\y} + \oa{\y\z} = \oa{\x\z} $ for any $ \x,\y,\z \in \M $. 
    \end{itemize}
\end{proposition}

An immediate property for the quasilinearized inner product is in Lemma \ref{lem:compare}, which is a consequence of the triangle comparison theorem of Toponogov; A simple comparison is presented below in Lemma \ref{lem:triangle-compare}. 

% The comparison theorems, that of Toponogov and that of \citet{rauch1951contribution}, yield a comparison of the lengths along geodesics in different Riemannian manifolds under suitable conditions.
% Applying Toponogov's comparison theorem, we can compare Hadamard manifold with its tangent space (\cite{cheeger1975comparison}). The two Riemannian manifolds satisfy the theorem's conditions, including initial condition and curvature's hypothesis and nonexistence of conjugate points. As a consequence, we get comparison about geodesic triangles in Hadamard manifold. 

\begin{lemma}[Triangle Comparison for Hadamard Manifolds]
    \label{lem:triangle-compare}
    Let $(\M; g)$ be a Hadamard manifold with non-positive sectional curvature. Then for any $\x \in \M$, any $\v_1 \in T_{\x}\M$ and $\v_2 \in T_{\x}\M$, one has
    \begin{align*}
        \| \v_1 - \v_2 \|_{\x} \le | \oa{\Exp_{\x}(\v_1) \Exp_{\x}(\v_2)} |.
    \end{align*}
\end{lemma}

\begin{proof}
    Apply Toponogov's comparison theorem to $(\M, g)$ and $( T_{\x}\M ,g_{\x})$.
\end{proof}

\begin{lemma} 
    \label{lem:compare} 
    Let $ (\M, g) $ be a Hadamard manifold with distance metric $d$. Then it holds that 
    \begin{align*} 
        \< \oa{\x\y} , \oa{\x\z} \> \le \< \Exp_{\x}^{-1} \( \y \) , \Exp_{\x}^{-1} \(\z\) \>_\x , \quad \forall \x, \y, \z \in \M . 
    \end{align*} 
\end{lemma} 

\begin{proof}
    By definition of the quasilinearized inner product, it suffices to prove 
    \begin{align}
        \cosq \( \oa{\x\y} , \oa{\x\z} \) 
        \le 
        \cos \measuredangle \( \x ; {\y}, \z \), \label{eq:target-0}
    \end{align}
    where $ \measuredangle \( \x ; {\y}, \z \) $ is the angle in $T_\x \M$ whose two sides map to $ \oa{\x\y} $ and $ \oa{ \x \z  } $ via the exponential map (at $\x$). Since $ T_y\M $ is flat, the law of cosine in Euclidean space gives 
    \begin{align} 
        \cos \measuredangle \( \x ; {\y}, \z \)
        = 
        \frac{ \| \Exp_{\x}^{-1} (\y) \|_\x^2 + \| \Exp_{\x}^{-1} (\z ) \|_\x^2 - \| \v \|_\x^2  }{ 2 \| \Exp_{\x}^{-1} (\y) \|_\x \| \Exp_{\x}^{-1} (\z ) \|_\x } , \label{eq:pause-0}
    \end{align} 
    where $ \v = \Exp_{\x}^{-1} \(\y\) - \Exp_{\x}^{-1} \(\z\) $, and the minus operation is defined in $T_\y \M$. By the triangle comparison Lemma \ref{lem:triangle-compare}, we know 
    \begin{align}
        \| \v \|_\x \le | \oa{ \y \z } |. 
        \label{eq:comp-0}
    \end{align}

    Plugging (\ref{eq:comp-0}) into (\ref{eq:pause-0}) gives 
    \begin{align*} 
        \cos \measuredangle \( \x ; {\y}, \z \) 
        =& \;  
        \frac{ \| \Exp_{\x}^{-1} (\y) \|_\x^2 + \| \Exp_{\x}^{-1} (\z ) \|_\x^2 - \| \v \|_\x^2  }{ 2 \| \Exp_{\x}^{-1} (\y) \|_\x \| \Exp_{\x}^{-1} (\z ) \|_\x } \\ 
        \ge& \;  
        \frac{ \| \Exp_{\x}^{-1} (\y) \|_\x^2 + \| \Exp_{\x}^{-1} (\z ) \|_\x^2 - | \oa{\y \z} |^2  }{ 2 \| \Exp_{\x}^{-1} (\y) \|_\x \| \Exp_{\x}^{-1} (\z ) \|_\x } \\ 
        =& \; 
        \frac{ | \oa{\x \y} |^2 + | \oa{\x \z} |^2 - | \oa{\y \z} |^2  }{ 2 | \oa{\x \y} | | \oa{\x \z} | } \\ 
        =& \; 
        \cosq \( \oa{\x\y} , \oa{\x\z} \) , 
    \end{align*} 
    which proves (\ref{eq:target-0}), and thus concludes the proof. 
    % and recalling the definition of $\cosq$ in (\ref{eq:def-inner}) finishes the proof. 
    % \end{align*}
    % --------------------------------------
    % Since $ | \y_{\oa{\grad f (\y)}} | = \| \grad f (\y) \|_\y $ and $ | \oa{\x\y} | = \| \Exp_{\y}^{-1} ( \x ) \|_\y $, it suffices to prove 
    % \begin{align*} 
    %     \cosq \( \y_{\oa{\grad f (\y)}} , \oa{\x\y} \) 
    %     \le 
    %     \cos \measuredangle \( \y ; {\x}, {\Exp_p (\grad f (\y))} \) , 
    % \end{align*} 
    % where $  \measuredangle \( \y ; {\x}, {\Exp_p (\grad f (\y))} \) $ is the angle in $T_\y \M$ whose two sides map to $ \oa{\y\x} $ and $ \oa{ \y \Exp_\y \(\grad f (\y)\) } $ via the exponential map (at $\y$). Since $ T_y\M $ is flat, the law of cosine in Euclidean manifold gives 
    % \begin{align} 
    %     \cos \measuredangle \( \y ; {\x}, {\Exp_p (\grad f (\y))} \)
    %     = 
    %     \frac{ \| \Exp_{\y}^{-1} (\x) \|_\y^2 + \| \grad f (\y) \|_\y^2 - \| \v \|_\y^2  }{ 2 \left\| \Exp_{\y}^{-1} (\x) \right\|_\y \| \grad f (\y) \|_\y } , \label{eq:pause}
    % \end{align} 
    % where $ \v = \Exp_{\y}^{-1} \(\x\) - \grad f (\y) $, and the minus operation is defined in $T_\y \M$. 
    % % is the vector in $ $
    % By the comparison theorem of Toponogov, we know 
    % \begin{align}
    %     \| \Exp_{\y}^{-1} \(\x\) - \grad f (\y) \| \le | \oa{ \x \z } |, 
    %     \label{eq:comp}
    % \end{align}
    % where $\z = \Exp_{\y} \( \grad f (\y) \)$. Plugging (\ref{eq:comp}) into (\ref{eq:pause}) and recalling the definition of $\cosq$ in (\ref{eq:def-inner}) finishes the proof. 
\end{proof}

% \section{Convexifying Effect of Quasi-linearized Inner Product} 

\subsection{Convexity over Hadamard manifold}

Now we present some basic facts about convex functions on Riemannian manifolds, which will serve as the preliminaries for this paper.
As a consequence of the celebrated results of Hopf--Rinow, the notion of geodesically convex functions can be defined over the entire Hadamard manifold; See Definition \ref{def:g-conv}. 

\begin{definition}[$g$-convexity]
    \label{def:g-conv}
    Let $ ( \M , g ) $ be a Hadamard manifold with distance metric $d$, and let $f$ be a real-valued differentiable function defined over $\M$. We say $ f $ is $g$-convex (or geodesically convex) with parameter $\mu$ if 
    \begin{align*}
        f (\x) \ge f (\y) + \< \grad f (\y) , \Exp_{\y}^{-1} ( \x ) \>_\y + \frac{\mu}{2} d^2 (\x, \y), \quad \forall \x, \y \in \M, 
    \end{align*} 
    where $ \< \cdot, \cdot\>_\y $ is the inner product in $ T_\y \M $. When $ \mu = 0$, we simply say $f$ is $g$-convex. When $ \mu > 0 $, we say $f$ is strongly $g$-convex with parameter $\mu$ (or $\mu$-strongly $g$-convex). 
    %When $ \mu < 0 $, we say $f$ is weakly $g$-convex with parameter $\mu$ (or $\mu$-weakly $g$-convex). 
    % Also, we use $ \| \cdot \|_\x $ to denote the norm in $ T_\x \M $ for any $\x \in \M$. 
    
    % for every $\x \in \M$, 
\end{definition}

In companion to the above notion of convexity, we introduce another notion of convexity, raised by the quasilinearized inner product. 

\begin{definition}[$q$-convexity]
    \label{def:q-conv}
    Let $ ( \M , g ) $ be a Hadamard manifold with distance metric $d$, and let $f$ be a real-valued differentiable function defined over $\M$. We say $ f $ is $q$-convex with parameter $\mu$ if
    \begin{align*} 
        f (\x) \ge f (\y) + \< \oa{\y \Exp_{\y} \( \grad f (\y) \) } , \oa{\y\x} \> + \frac{\mu}{2} d^2 (\x, \y) , \quad \forall \x, \y \in \M, 
    \end{align*}  
    where $ \< \cdot, \cdot\> $ is defined in (\ref{eq:def-inner}). When $ \mu = 0$, we simply say $f$ is $q$-convex. 
    %When $ \mu > 0 $, we say $f$ is strongly $q$-convex. When $ \mu < 0 $, we say $f$ is weakly $q$-convex.
    % for every $\x \in \M$, 
\end{definition}

We have the following result that relates these two notions of convexity. 

%     The above inner product defined in (\ref{eq:def-inner}) generalizes almost all calculation rules of the Euclidean inner product. 
% \end{claim} 

\begin{lemma}
    \label{lem:conv-rela-1}
    If function $f$ is $g$-convex, then $f$ is $q$-convex. 
\end{lemma}

\begin{proof} 
    It suffices to prove, for any $\x, \y \in \M$, 
    \begin{align} 
         \< \oa{\y \Exp_{\y} \( \grad f (\y) \) } , \oa{\y\x} \> 
         \le 
         \< \grad f (\y) , \Exp_{\y}^{-1} (\x ) \>_\y . \label{eq:key}
    \end{align} 
    Eq.(\ref{eq:key}) follows directly from Lemma \ref{lem:compare}.
    % Since $ | \oa{ \y \Exp_\y (\grad f (\y)) } | = \| \grad f (\y) \|_\y $ and $ | \oa{\x\y} | = \| \Exp_{\y}^{-1} ( \x ) \|_\y $, it suffices to prove 
    % \begin{align*} 
    %     \cosq \( \oa{ \y \Exp_\y (\grad f (\y)) } , \oa{\x\y} \) 
    %     \le 
    %     \cos \measuredangle \( \y ; {\x}, {\Exp_p (\grad f (\y))} \) , 
    % \end{align*} 
    % where $  \measuredangle \( \y ; {\x}, {\Exp_p (\grad f (\y))} \) $ is the angle in $T_\y \M$ whose two sides map to $ \oa{\y\x} $ and $ \oa{ \y \Exp_\y \(\grad f (\y)\) } $ via the exponential map (at $\y$). Since $ T_\y \M $ is flat, the law of cosine in Euclidean manifold gives 
    % \begin{align} 
    %     \cos \measuredangle \( \y ; {\x}, {\Exp_p (\grad f (\y))} \)
    %     = 
    %     \frac{ \| \Exp_{\y}^{-1} (\x) \|_\y^2 + \| \grad f (\y) \|_\y^2 - \| \v \|_\y^2  }{ 2 \left\| \Exp_{\y}^{-1} (\x) \right\|_\y \| \grad f (\y) \|_\y } , \label{eq:pause}
    % \end{align} 
    % where $ \v = \Exp_{\y}^{-1} \(\x\) - \grad f (\y) $, and the minus operation is defined in $T_\y \M$. 
    % % is the vector in $ $
    % By the comparison theorem of Toponogov, we know 
    % \begin{align}
    %     \| \Exp_{\y}^{-1} \(\x\) - \grad f (\y) \| \le | \oa{ \x \z } |, 
    %     \label{eq:comp}
    % \end{align}
    % where $\z = \Exp_{\y} \( \grad f (\y) \)$. Plugging Eq. (\ref{eq:comp}) into Eq. (\ref{eq:pause}) and recalling the definition of $\cosq$ in Eq. (\ref{eq:def-inner}) finishes the proof. 
\end{proof}

\section{First-order methods for convex optimization over Hadamard manifolds}

\label{section:method}

% In the next section, we briefly introduce our algorithm for geodesically convex optimization over Hadamard manifolds.
% We outline the intuition behind the algorithm and formally present our main results. A detailed exposition of the algorithmic framework and supporting proofs will be provided in subsequent sections.

% \subsection{Proximal Iteration} 
% Before introducing the algorithm, 

% The 
% We first present the iterative update rule.
% We define the proximal iteration sequence $\{ \x_t \}_t$ by
The first-order method we will investigate is given by the following Proximal Gradient iteration: 
\begin{align}
    \x_t = \Exp_{\x_{t+1}} \( \eta \grad f (\x_{t+1}) \). 
    \label{eq:proximal-iter}
\end{align} 
% This implicit proximal update inherents several important properties from its Euclidean counterparts, as formalized in Proposition \ref{??}
This implicit proximal update preserves several important properties of its Euclidean analogues, as stated in Proposition \ref{prop:proximal}.

\begin{proposition} 
\label{prop:proximal} 
    Let $(\M, g)$ be a Hadamard manifold with distance metric $d$. Let $f$ be a $g$-convex and differential function defined over $\M$. 
    % Then the following holds. 
    % \begin{itemize}
        % \item[Equivalence]
        % \item[Efficient ]
    % \end{itemize}
    Then for any $\x \in \M$ and $\eta > 0$, $\y$ solves 
    \begin{align*} 
        \x = \Exp_{\y} \( \eta \grad f (\y) \)
    \end{align*} 
    if and only if $\y$ solves 
    \begin{align*} 
        \y = \arg\min_\z \left\{ f (\z) + \frac{1}{2\eta}  d(\x,\z)^2 \right\} . 
    \end{align*}  
\end{proposition} 

\begin{proof}
% \begin{proposition}
% \label{gradsqd}
    % Let $(\M, d)$ be a Hadamard manifold. 
    For fixed $\x \in \M$, let $d_\x^2 (\z) := (d (\x, \z))^2$ be the squared distance, seen as a function of $\z \in \M$. We will first prove the following equality:
    % then its gradient is:
    \begin{align}
      \grad  d_\x^2 (\z) = - 2 \Exp_\z^{- 1} (\x). \label{eq:sq-dist}
    \end{align}
% \end{proposition}
    By the celebrated Gauss's Lemma, we can adopt a geodesic polar coordinate system near $\x$:
    \begin{equation} 
    d s^2 = d r^2 + \sum_{i, j} g_{i j} (r, \theta) d \theta^i d \theta^j ,
    \label{polar}
    \end{equation}
    where $r (\z) = d (\x, \z)$ and $\theta = (\theta^1, \ldots, \theta^{n - 1})$
    is a coordinate on the geodesic hypersphere $\mathbb{S}_\x (r) = \{ \Exp_\x (\v) | \v \in T_\x \M \|, \| \v \| = r \}$. Moreover, by Eq. (\ref{polar}), 
    \begin{align*}
      \left\langle \frac{\partial}{\partial r}, \frac{\partial}{\partial r}
     \right\rangle_{\x} = 1, 
     \quad \left\langle \frac{\partial}{\partial r},
     \frac{\partial}{\partial \theta^i} \right\rangle_{\x} = 0.
    \end{align*}
    Thus $\left. \frac{\partial}{\partial r} \right|_z \in T_z M$ is the normal
    vector in the radial direction. More accurately, consider the geodesic
    $\gamma : [0, 1] \rightarrow M$, $\gamma (0) = x, \gamma (1) = z$. Gauss's
    Lemma tells us $\left\langle \gamma' (t), \frac{\partial}{\partial \theta^i}
    \right\rangle = 0$, thus $\frac{\gamma' (t)}{\| \gamma' (t) \|} = \left.
    \frac{\partial}{\partial r} \right|_{\gamma (t)}$. The time-reversed
    $\overline{\gamma} (t) = \gamma (1 - t)$ is the geodesic from
    $\overline{\gamma} (0) = \z$ to $\overline{\gamma} (1) = \x$, so by the
    definition of exponential map, $\Exp_\z^{- 1} (\x) = \overline{\gamma}'
    (0) = - \gamma' (1)$. Thus we get
    \begin{align*}
    \left. \frac{\partial}{\partial r} \right|_\z = - \frac{\Exp_\z^{- 1}
     (\x)}{\| \Exp_\z^{- 1} (\x) \|} . 
    \end{align*}
    Let us compute $\text{grad} r$. By definition, $\text{grad} r \in
    \mathfrak{X} (M)$ is a smooth vector field such that
    \begin{align*}
    \forall X \in \mathfrak{X} (M), \quad \langle \text{grad} r, X \rangle = d r
     (X) = X r.
    \end{align*}
    In the geodesic polar coordinate system, $X = \eta \frac{\partial}{\partial
    r} + \sum_i \xi^i \frac{\partial}{\partial \theta^i}$. Then $X r = \eta =
    \left\langle \frac{\partial}{\partial r}, X \right\rangle$. Thus
    \begin{align*}
      \text{grad} r = \frac{\partial}{\partial r} .
    \end{align*}
    Then
    \begin{align*}
    \text{grad} r^2 = 2 r \text{grad} r = 2 r \frac{\partial}{\partial r} .
    \end{align*}
    Now $d_\x^2 (\z) = (d (\x, \z))^2 = r^2 (\z)$ and $\left.
    \frac{\partial}{\partial r} \right|_\z = - \frac{\Exp_\z^{- 1} (\x)}{\|
    \Exp_\z^{- 1} (\x) \|}$ gives
    \begin{align*}
    \text{grad} d_\x^2 (\z) = (\text{grad} r^2) |_\z  = 2 r (\z) 
     \left. \frac{\partial}{\partial r} \right|_\z = - 2 d (\x, \z)
     \frac{\Exp_\z^{- 1} (\x)}{\| \Exp_\z^{- 1} (\x) \|} = - 2
     \Exp_\z^{- 1} (\x) .
    \end{align*}

    \textbf{The \textit{if} direction.} If $\y$ solves $\y = \arg\min_\z \left\{ f (\z) + \frac{1}{2\eta}  d(\x,\z)^2 \right\}$, then $ \grad f (\y) + \frac{1}{2\eta} \grad \; d_{\x}^2 (\y) = 0 $ since $f$ is convex differentiable. By \eqref{eq:sq-dist} we have $ \x = \Exp_{\y} \( \eta \grad f (\y) \) $. 

    \textbf{The \textit{only if} direction.} If $\y$ solves $ \x = \Exp_{\y} \( \eta \grad f (\y) \) $, then by \eqref{eq:sq-dist}, $ \grad f (\y) + \frac{1}{2\eta} \grad \; d_{\x}^2 (\y) = 0 $. Since $f$ is convex differentiable, first-order optimality condition implies $\y$ solves 
    \begin{align*}
        \y = \arg\min_\z \left\{ f (\z) + \frac{1}{2\eta}  d(\x,\z)^2 \right\} . 
    \end{align*}
    % $ \y = \arg\min_\z \left\{ f (\z) + \frac{1}{2\eta}  d(\x,\z)^2 \right\} $. 
    % By \eqref{eq:sq-dist} we have $ \x = \Exp_{\y} \( \eta \grad f (\y) \) $. 
    
    % The proximal update rule gives 
    % \begin{align*} 
    %      \x = \Exp_{\y} \( \eta \grad f (\y) \) , \quad \text{ and thus } \quad  \eta \grad f (\y) = \Exp_{\y}^{-1} \(\x\).  
    % \end{align*} 

    % By Proposition \ref{gradsqd}, this equals to  
    % \begin{align*} 
    %     \grad f (\y) + \frac{1}{2\eta} \grad \; d_{\x}^2 (\y) = 0 . 
    % \end{align*} 
    % For $g$-convex function $f$, this means
    % \begin{align*}
    %     \y = \arg\min_\z \left\{ f (\z) + \frac{1}{2\eta}  d(\x,\z)^2 \right\}. 
    % \end{align*}
\end{proof}

This formulation of the proximal operator facilitates an efficient implementation, detailed in Section \ref{section:proximal}. 
%The proximal algorithm is summarized in Algorithm \ref{alg:prox}.
We present our main convergence rate in Theorem \ref{thm:main}, which admits a concise proof via quasilinearization.

% \begin{algorithm}[htb] 
%     \caption{Proximal Gradient Algorithm} 
%     \label{alg:prox} 
%     \begin{algorithmic}[1]  
%         \STATE \textbf{Input:} step size $\eta$; starting point $\x_0$.
%         \FOR{$t=0,1,\cdots$}
%             \STATE solve $\x_{t+1}$ by $ \x_t = \Exp_{\x_{t+1}} \( \eta \grad f (\x_{t+1}) \) $. 
%         \ENDFOR 
%     \end{algorithmic} 
% \end{algorithm} 

% This characterization of the proximal operation allows us to implement it efficiently. Next, we state the convergence result, whose proof is immediate with the notion of quasilinearization. 

\begin{theorem} 
    \label{thm:main}
    Let $ ( \M , g ) $ be a Hadamard manifold with distance metric $d$, and let $f$ be a real-valued differentiable function defined over $\M$.
    Let the objective $f$ be $g$-convex on Hadamard manifold. 
    The sequence $\{\x_t \}_t$ governed by Proximal Gradient Algorithm ( Eq.(\ref{eq:proximal-iter})) satisfies 
    \begin{align*} 
        f (\x_t) - f (\x^*) 
        \le 
        \frac{|\oa{\x_0\x^*}|^2 }{\eta t} , \quad \forall t . 
    \end{align*} 
    \label{thm:prox}
\end{theorem}

% Our method achieves a sublinear convergence rate, consistent with gradient descent and other first-order convex optimization methods. However, our results apply to arbitrary Hadamard manifolds, without requiring bounded curvature or a bounded domain. 
% To prove Theorem \ref{thm:prox}, we employ a key technique called the quasilinearization. A detailed proof will be provided in Section \ref{sec:proof}.

% \subsection{Proof of Theorem \ref{thm:prox}}

% \label{sec:proof}

% Lemma \ref{lem:compare} compares two types of inner products and plays a crucial role in the proof of Theorem \ref{thm:prox}. The proof is presented as follows.

\begin{proof}[Proof of Theorem \ref{thm:main}]
    The proximal update rule in Eq. (\ref{eq:proximal-iter}) gives 
    \begin{align*} 
        \eta \grad f (\x_{t+1}) = \Exp_{\x_{t+1}}^{-1} \(\x_{t}\) . %\quad \text{ and } \quad \x_t = \Exp_{\x_{t+1}} \( \eta \grad f (\x_{t+1}) \) .  
    \end{align*} 
    
    By $g$-convexity, we know 
    \begin{align*} 
        \; f (\x_{t+1}) - f (\x^*) 
        \le &
        \< - \grad f (\x_{t+1}) , \Exp_{\x_{t+1}}^{-1} \(\x^*\) \>_{\x_{t+1}} \\ 
        \overset{(i)}{\le}& \; 
        - \< \oa{ \x_{t+1} \Exp_{\x_{t+1}} \( \grad f (\x_{t+1}) \) } , \oa{ \x_{t+1} \x^* } \> 
        =
        \< \oa{ \Exp_{\x_{t+1}} \( \grad f (\x_{t+1}) \) \x_{t+1} } , \oa{ \x_{t+1} \x^* } \> , 
        % = 
        % \label{eq:convex} 
    \end{align*} 
    where $(i)$ follows from Lemma \ref{lem:compare}. 
    
    By the update rule, we have $\x_{t} = \Exp_{\x_{t+1}} \( \eta \grad f (\x_{t+1})\) $. 
    Therefore, 
    \begin{align} 
        \eta \( f (\x_{t+1}) - f (\x^*) \)
        \le 
        \< \oa{ \x_{t} \x_{t+1}  } , \oa{ \x_{t+1} \x^* } \> 
        = 
        - \frac{\eta^2}{2}  \| \grad f (\x_{t+1}) \|_{\x_{t+1}}^2 - \frac{1}{2} | \oa{\x_{t+1} \x^*} |^2 + \frac{1}{2} | \oa{\x_{t} \x^*} |^2 .  \label{eq:for-tele}
    \end{align} 

    By the equivalent characterization of the proximal operator in Proposition \ref{prop:proximal}, we obtain 
    \begin{align*}
        f (\x_{t+1}) = \min_{\y} \left\{ f(\y) + \frac{1}{2 \eta} d(\x_t, \y)^2 \right\}, 
    \end{align*}
    and thus
    \begin{align*} 
        f (\x_{t+1}) \le f (\x_t). 
    \end{align*} 

    Then telescoping Eq. (\ref{eq:for-tele}) and summing over $t$ give 
    \begin{align*} 
        f (\x_t) - f (\x^*) 
        \le
        \frac{1}{t} \sum_{i=1}^t f(\x_i) - f(\x^*)
        \le
        \frac{1}{t} \sum_{i=1}^t \left( - \frac{1}{2\eta} | \oa{\x_{i} \x^*} |^2 + \frac{1}{2\eta} | \oa{\x_{i-1} \x^*} |^2 \right)
        \le
        \frac{|\oa{\x_0\x^*}|^2 }{\eta t} , \quad \forall t . 
    \end{align*} 
\end{proof}

\subsection{Stochastic First-order Method for $g$-Convex Objectives} 
\label{sec:stoc-convex}

In this section, we extend the proximal iteration to the stochastic optimization setting, and present the corresponding algorithm along with its theoretical guarantees.

In stochastic optimization, we aim to analyze 
$$F(\x) = \E [ f(\x; \xi) ] = \int_{\xi \in \Theta } f (\x; \xi) \; \lambda (d\xi), $$ 
where $\lambda$ is a measure over $\Theta$. A concrete example of such objectives is the finite-sum objectives that are prevalent in machine learning. 
In such cases, we cannot directly observe information about $F$. Instead, we only have access to stochastic samples $f(\x; \xi)$ and its gradient $\grad f(\x; \xi)$, where $\xi$ is sampled from the law of $\lambda$.

The next assumption limits the variance of the random variables $\xi$. Constraints like this are commonly adopted in the literature, and necessary to keep the variance under control.

% Consider the following (finite-sum) optimization problem: 
% \begin{align*} 
%     \min_{\x \in \M} F (\x),  
% \end{align*} 
% where $ F (\x) = \int_{\xi} f (\x; \xi) \mu (d\xi) .  $

\begin{assumption}
    Suppose that $ f (\cdot; \xi) $ is $g$-convex for any $\xi$. Also, instate a bounded variance assumption: 
    \begin{align*}
        \E \| \grad  f (\x; \xi) \|_\x^2 \le \| \grad F (\x) \|_{\x}^2 + \sigma^2 , \quad \forall \x . 
    \end{align*}
    \label{assump:finite-var}
\end{assumption}

% \textbf{Algorithm procedure:} 
% In each step $ t $, we observe a realization of the random variable $\xi$, denoted $\xi_t$, and perform a proximal gradient step on the function $ f (\x; \xi_t) $. This is an abstraction of batch minimization for finite-sum problems, where the objective is $ F (\x) := \sum_{i=1}^n f_i (\x) $, and each time we randomly sample a subset of $\{ f_i \}_{i=1}^n$ to perform a first-order step. 

Proximal operator can be used in stochastic setting. The algorithm is outlined below, with convergence rate provided in Theorem \ref{thm:stoc-pro}.

\begin{algorithm}[htb] 
    \caption{Stochastic Proximal Gradient Method} 
    \label{alg:prox-stoc} 
    \begin{algorithmic}[1]  
        \STATE \textbf{Input:} step sizes $\{\eta_t\}_t$; starting point $\x_0$.
        \FOR{$t=1,2,\cdots$} 
            \STATE Pick $ \xi_t $ governed by the law $\mu$. 
            \STATE \text{ solve } $\x_{t+1}$ \text{by} $ \x_t = \Exp_{\x_{t+1}} \( \eta_t \grad f(\x_{t+1}; \xi_t) \) $
            %$ \x_{t+1} = \mathrm{Prox}_{\eta_t,\x_t} f(\cdot; \xi_t) $. 
        \ENDFOR 
    \end{algorithmic} 
\end{algorithm} 

\begin{theorem} 
    \label{thm:stoc-pro}
    Let $(\M,g)$ be a Hadamard manifold with distance metric $d$. Let Assumption \ref{assump:finite-var} be true. Function $F$ is strongly $g$-convex with parameter $\mu$ and $L$-smooth. For Stochastic Proximal Gradient Algorithm, if step size $\eta_t = \frac{1}{2L \sqrt{t}}$, it satisfies 
    \begin{align}
        \frac{1}{\sum_{t=1}^T \alpha_t} \sum_{t=1}^T \alpha_t \( \E  F (\x_t) - F (\x^* ) \) 
        \le 
        \frac{| \oa{\x_{0} \x^*} |^2}{2 \sum_{t=1}^T \alpha_t }  +  \frac{  \sigma^2 \log (T+1) }{ \sum_{t=1}^T \alpha_t } , \label{eq:rate-1}
    \end{align}
    where $\alpha_t = \frac{1}{4 L \sqrt{t}}$ for all $t$. 
    % If $ \eta_t \approx \frac{1}{\sqrt{t}} $, then the algorithm (a weighted average of the trajectory) converges at rate $ O \( \frac{\log T}{ \sqrt{T} } \) $ in expectation. 

    If $ \eta_t = \frac{1}{\sqrt{t} \log t} $, then 
    \begin{align}
        \sum_{s=1}^\infty \( \eta_s - L \eta_s^2\) \( F (\x_s) - F (\x^*) \) < \infty, \quad a.s. . \label{eq:rate-2} 
    \end{align} 
    % where $\beta_t := $
\end{theorem} 

\begin{proof} 
    Consider the stochastic proximal gradient algorithm: 
    \begin{align*}
        \x_t = \Exp_{\x_{t+1}} \( \eta_t \grad f(\x_{t+1}; \xi_t) \).
    \end{align*}
    % \begin{align*} 
    %     \x_{t+1} 
    %     = 
    %     \mathrm{Prox}_{\eta_t,\x_t} f (\cdot; \xi_t) . 
    %     % := \arg\min \left\{ f (\y) + \frac{1}{2\eta} | \oa{\x\y} |^2 \right\} 
    %     % = \x_t - \eta_t \nabla f ( \x_{t+1} ; \xi_t ) . 
    % \end{align*} 
    
    Then combined with Lemma \ref{lem:compare}, we have 
    \begin{align*} 
        f ( \x_{t+1} ; \xi_t ) - f ( \x^* ; \xi_t ) 
        \le 
        \< \oa{ \Exp_{\x_{t+1}} \( \grad f (\x_{t+1}; \xi_t ) \) \x_{t+1} } , \oa{ \x_{t+1} \x^* } \> , 
        % = 
        % \frac{1}{2} \( \| \oa{\x_t} \| \)
    \end{align*}
    which implies 
    \begin{align} 
        \eta_t \( f ( \x_{t+1} ; \xi_t ) - f ( \x^* ; \xi_t ) \) 
        \le 
        - \frac{1}{2} \eta_t^2  \| \grad f (\x_{t+1}; \xi_t) \|_{\x_{t+1}}^2 - \frac{1}{2} | \oa{\x_{t+1} \x^*} |^2 + \frac{1}{2} | \oa{\x_{t} \x^*} |^2. \label{eq:descent-tri}
        % \le 
        % - | \oa{\x_{t+1} \x^*} |^2 + | \oa{\x_{t} \x^*} |^2 . 
    \end{align} 
    
    % Therefore, we have 
    % \begin{align*} 
    %     \eta_t \E_{\xi_t^k} \[ f ( \x_{t+1} ; \xi_t ) - f ( \x^* ; \xi_t ) \] 
    %     \le 
    %     | \oa{\x_{t} \x^*} |^2 - \E_{\xi_t^k} \[ | \oa{\x_{t+1} \x^*} |^2 \] . 
    % \end{align*} 
    
    % We can now again telescope. 
    
    % -----------------------------------------------------
    
    Since $ f (\cdot; \xi) $ is convex for any $\xi$, we have 
    \begin{align*} 
        \eta_t f (\x_{t+1}; \xi_t ) 
        \ge& \;  
        \eta_t f (\x_t ;\xi_t ) + \< \eta_t \grad f (\x_t; \xi_t ), \Exp_{\x_t}^{-1} \(\x_{t+1}\) \>_{\x_t} \\ 
        \ge& \;  
        \eta_t f (\x_t ;\xi_t ) - \frac{\eta_t^2}{2} \| \grad f (\x_t; \xi_t ) \|_{\x_t}^2 - \frac{1}{2} | \oa{\x_t \x_{t+1}} |^2 , 
    \end{align*} 
    where the last line uses Young's inequality. 
    
    Combining the above inequality with (\ref{eq:descent-tri}) gives 
    \begin{align} 
        \eta_t \( f (\x_t; \xi_t) - f (\x^* ; \xi_t ) \) 
        \le 
        - \frac{1}{2} | \oa{\x_{t+1} \x^*} |^2 + \frac{1}{2} | \oa{\x_{t} \x^*} |^2 + \frac{\eta_t^2}{2} \| \grad f (\x_t; \xi_t) \|_{\x_t}^2 . 
        \label{eq:intermed}
    \end{align}
    We can telescope the above inequality and obtain 
    \begin{align*}
        \sum_t \eta_t \( f (\x_t; \xi_t) - f (\x^* ; \xi_t ) \) 
        \le 
        \frac{1}{2} | \oa{\x_{0} \x^*} |^2 + \sum_t \frac{\eta_t^2}{2} \| \grad f (\x_t; \xi_t) \|_{\x_t}^2. 
    \end{align*}
    
    Taking expectation on both sides of the above inequality yields 
    \begin{align*}
        \sum_t \eta_t \( \E  F (\x_t) - F (\x^* ) \)
        \le 
        \frac{1}{2} | \oa{\x_{0} \x^*} |^2 + \sum_t \frac{\eta_t^2}{2} \E \| \grad F (\x_t) \|_{\x_t}^2 + \sum_t \frac{\eta_t^2 \sigma^2 }{2} , 
    \end{align*}
    Here we apply property of the $L$-smoothness to the above inequality, which shows 
    \begin{align*}  
        \sum_t \( \eta_t - L \eta_t^2 \) \( \E  F (\x_t) - F (\x^* ) \) 
        \le 
        \frac{1}{2} | \oa{\x_{0} \x^*} |^2 + \sum_t \frac{\eta_t^2 \sigma^2 }{2} . 
    \end{align*} 
    
    Since $ \eta_t - L \eta_t^2 \ge \frac{1}{4L \sqrt{t}} $ when $ \eta_t = \frac{1}{2L \sqrt{t}} $, the above inequality with this choice of $\eta_t$ gives 
    \begin{align*} 
        \sum_{t=1}^T \frac{1}{4L \sqrt{t}} \( \E  F (\x_t) - F (\x^* ) \) 
        \le 
        \frac{1}{2} | \oa{\x_{0} \x^*} |^2 + \frac{\sigma^2 \log (T+1)}{32L^2}, \quad \forall T \ge 2 , 
        % \sum_t \frac{\eta_t^2 \sigma^2 }{2} . 
    \end{align*} 
    which proves (\ref{eq:rate-1}) after rearranging terms. 
    % rearranging terms finishes the proof of the first inequality. 

    Next we turn to the proof of (\ref{eq:rate-2}). For this, we consider the following random variables: 
    \begin{align*} 
        V_t := \sum_{s=1}^t \( \eta_s - L \eta_s^2 \) \( F (\x_s) - F (\x^*) \) , 
    \end{align*}

    With this definition of $V_t$, we have 
    \begin{align*} 
        \( \eta_{t} - L \eta_t^2 \) \( F ( \x_t ) - F (\x^*) \) 
        = 
        V_t - V_{t-1} . 
        % \E \[ V_{t} | \x_t \] - V_{t-1} 
        % \le 
        % \frac{\eta_t^2 \sigma^2 }{2} . 
    \end{align*} 

    Now we take (conditional) expectation on (\ref{eq:intermed}), and apply the $L$-smoothness property to obtain 
    \begin{align*} 
        \E \[ V_t | \xi_1, \cdots, \xi_{t-1} \] - V_{t-1}
        =& \;  
        \( \eta_t - L \eta_t^2 \) \( \E \[ F (\x_t) | \xi_1, \cdots, \xi_{t-1} \] - F (\x^*) \) \\ 
        \le& \;  
        \frac{\sigma^2 \eta_t^2 }{2} - \frac{1}{2} \E \[ | \oa{\x_{t+1} \x^*} | \big| \xi_1, \cdots, \xi_{t-1} \] + | \oa{\x_t \x^*} | . 
    \end{align*} 

    Let $ \beta_t := \frac{\sigma^2 \eta_t^2 }{2} - \frac{1}{2} \E \[ | \oa{\x_{t+1} \x^*} | \big| \xi_1, \cdots, \xi_{t-1} \] + | \oa{\x_t \x^*} | $, which is adapted to the random sequence $ \{ \xi_t \}_t $. Now we pick $\eta_t = \frac{1}{\sqrt{t} \log t} $, and thus $ \sum_{t=1}^\infty \beta_t < \infty $. This shows that $ \{ V_t \} $ is a almost super-martingale (adapted to $\{ \xi_t \}_t$). 
    
    By 
    % Let $ \eta_t = \frac{1}{\sqrt{t} \log t} $. Then we have $ \sum_{t=1}^\infty \frac{\eta_t^2 \sigma^2 }{2} < \infty $, and by 
    the Robbins--Siegmund theorem, we know $ \{ V_t \}_t $ converges almost surely. Therefore, 
    \begin{align*} 
        \sum_{s=1}^\infty \( \eta_s - L \eta_s^2 \) \( F (\x_s) - F (\x^*) \) 
    \end{align*} 
    exists with probability 1. This means $\sum_{s=1}^\infty \( \eta_s - L \eta_s^2 \) \( F (\x_s) - F (\x^*) \) < \infty $ for almost all points $\omega$ in the sample space $\Omega$. This concludes the proof. 
\end{proof}

%-----------

\section{The Proximal Operator}
\label{section:proximal}

Building upon the formal derivation of the proximal operator and its key properties in Proposition \ref{prop:proximal}, we now present supplementary details regarding its practical implementation.

\subsection{Proximal Operator on Manifolds}

% The proximal operator plays a central role in the algorithms above. 
On a manifold $\M $, the proximal operator for a function $ f $ at a point $\x \in \M $ with parameter $ \eta > 0 $ is defined as $\mathrm{Prox}_{\eta,\x} f$. The point $\y := \mathrm{Prox}_{\eta,\x} f$ is the solution to the equation:
\begin{align}
    \x = \Exp_{\y} \( \eta  \grad f \( \y \) \).
\end{align}

By Proposition \ref{prop:proximal}, a equivalent definition of the proximal operator is given by: 
\begin{align}
    \mathrm{Prox}_{\eta,\x} f := \arg\min_\y \left\{ f (\y) + \frac{1}{2\eta}  d(\x,\y)^2 \right\}. 
\label{eq:prox}
\end{align}

This operator seeks the minimizer of the regularized objective function:
\begin{align}
h_{\eta,\x}(\y) := f (\y) + \frac{1}{2\eta}  d(\x,\y)^2.
\label{eq:prox-h}
\end{align}
Since $ d(\x,\y)^2 $ is strongly convex, the function $ h_{\eta,\x} $ inherits strong convexity when $ f $ is convex. 
% This property is crucial for the convergence analysis of the algorithm.

% By combining Proposition \ref{prop:proximal} above with Definition \ref{def:quail-inner-prod}, we derive the following inner product equality:

% \begin{proposition} 
% \label{prop:proximal-inner} 
%     Let $(\M, g)$ be a Hadamard manifold with distance metric $d$. For a function $f$ defined over $\M$, let $\x \in \M$, $\eta > 0$. Set $\y := \mathrm{Prox}_{\eta,\x} f$ then:
%     \begin{align*} 
%       \< \oa{\x \x^*}, \oa{\x \y} \> 
%       =
%       \frac{1}{2} | \oa{\x\x^*} |^2 - \frac{1}{2} | \oa{\y\x^*} |^2 + \frac{\eta^2}{2} \| \grad f(\y) \|_{\y}^2.  
%     \end{align*} 
% \end{proposition} 

% Unlike Proposition \ref{thm:gd-compare}, this result requires no extra assumptions on the Hadamard manifold while retaining the same key advantage --- it establishes a relationship between the current iterate, the subsequent iterate (of the proximal step), and the optimal point. 
% Specifically, Proposition \ref{prop:proximal-inner} allows $\M$ to be unbounded, and the lower bound of its curvature is allowed to tend to infinity.

\subsection{Implementing the Proximal Operator}

If the proximal operator in algorithms cannot be computed in closed form, one may employ fixed-point iteration or other numerical methods to approximate it. For instance, one may apply gradient descent to minimize $f (\y) + \frac{1}{2\eta}  d(\x,\y)^2$, which is strongly $g$-convex. Building on the convergence analysis of convex (\emph{but not globally smooth}) optimization \citep[e.g.,][]{boyd2004convex,nesterov2018lectures}, we can show that the proximal operator can be efficiently implemented. 

% Fix $\x$, our goal is to find $\y_0 := \mathrm{Prox}_{\eta,\x} f$ such that
% \begin{align*}
%     \Exp_{\y_0} \( \eta  \grad f \( \y_0 \) \) -\x = 0.
% \end{align*}
% We define a mapping $T$ from $\mathcal{M} $ to $\mathbb{R}$ : 
% \begin{align*}
%     T ( \y ) = \Exp_{\y} \( \eta  \grad f \( \y \) \) -\x + \y.
% \end{align*}
% Then we adopt the fixed point iteration for solving equation $T(\y)= \y$. The process would generate the iteration:
% \begin{align*}
%     \y_{t+1} = T( \y_t ).
% \end{align*}
% Its convergence $\y_t \rightarrow \y_0$ can be guaranteed by the properties of the fixed-point iteration, where $\y_0$ is the fixed point such that $T(\y_0)= \y_0$.

% \begin{proposition}
%     If the mapping $T$ is $\alpha$-Lipschitz continuous with $\alpha \in [0,1)$, then it has an unique fixed point $\y_0$ and the fixed point iteration converges to $\y_0$ linearly:
%     \begin{align*}
%         \| T(\y_t) - T(\y_0) \| \le \alpha^t d \( \y_1, \y_0 \).
%     \end{align*}
% \end{proposition}

%--------------------

% This approach generalizes the classical gradient descent method to the setting of Riemannian optimization, where the exponential map ensures that the updates remain on the manifold.

% In Theorem \ref{thm:strong-conv}, we present a convergence guarantee for Algorithm \ref{alg:gd} -- a linear convergence rate for gradient descent when applied to strongly $g$-convex and $L$-smooth functions on a Hadamard manifold. 

We now proceed to analyze gradient descent for optimizing strongly $g$-convex functions. 
% , beginning with the determination of the smoothness parameter $L_0$.
Consider a function $f$ that is $\mu$-strongly $g$-convex, due to strong convexity, distance between starting point $\z_0$ and unique minima $\z^*$ is bounded by
\begin{align*}
    d^2 (\z_0, \z^*) \le \frac{2}{\mu} \( f(\z_0) - f(\z^*) \).
\end{align*}
We define $\D$ as the sublevel set for $\z_0$:
\begin{align*}
    \D := \left\{ x : d (\z, \z_0) \le \sqrt{\frac{2}{\mu} \( f(\z_0) - f(\z^*) \)} \right\}, 
\end{align*}
which is compact due to strong convexity of $f$. 
If $f$ is continuously differentiable, the quantity $\max_{\z \in \D} \|  \grad f(\z) \|$ is bounded, due to compactness of $\D$. We define $\D_0$ as set: 
\begin{align*}
    \D_0 := \left\{ \z : d (\z, \z_0) \le \sqrt{\frac{2}{\mu} \( f(\z_0) - f(\z^*) \)} + \frac{1}{\mu} \max_{\z \in \D} \|  \grad f(\z) \|_\z \right\}, 
\end{align*}
which is also compact. 
Again, when $f$ is continuously differentiable, there exist a constant $L_0 > 0$ such that function $f$ restricted on $\D_0$ is $L_0$-smooth. Specifically, we define $L_0$ as 
\begin{align*}
    L_0 := \max_{\z \in \D_0} \{ \| \grad f(\z) \|  \} .
\end{align*}

The gradient descent algorithm is outlined in Algorithm \ref{alg:gd}. In Theorem \ref{thm:strong-conv}, we present a convergence guarantee for Algorithm \ref{alg:gd}.

\begin{algorithm}[h!] 
    \caption{Gradient Descent} 
    \label{alg:gd} 
    \begin{algorithmic}[1]  
        \STATE \textbf{Input:} step size $\eta$; starting point $\z_0$.
        \FOR{$t=0,1,\cdots$}
            \STATE $ \z_{t+1} = \Exp_{\z_t} \( -\eta \grad f (\z_t)\) $. 
        \ENDFOR 
    \end{algorithmic} 
\end{algorithm} 

\begin{theorem}
    \label{thm:strong-conv}
    Let $(\M,g)$ be a Hadamard manifold with distance metric $d$. Let $f$ be strongly $g$-convex with parameter $\mu$, and continuously differentiable. Then the Gradient Descent algorithm (Algorithm \ref{alg:gd}) with step size $\eta = \frac{1}{L_0} $ satisfies 
    \begin{align} 
        f (\z_t) - f (\z^*) 
        \le 
        \( 1 - \frac{\mu}{4 L_0} \)^{t} 
        \( f (\z_0) - f (\z^*) \) , \label{eq:thm1}
    \end{align} 
    where $\z^*$ is the unique minimizer of $f$, and $L_0$ is defined as above. 
\end{theorem} 

\begin{proof} 
    We first show by induction that $\z_{t} $ belong to $\D_0$. The initial $\z_0 \in \D \subset \D_0$ holds by definition. Assume that $\z_t \in \D$.  Since $ \z_{t+1} = \Exp_{\z_t} \( -\eta \grad f (\z_t)\) $, distance $| \oa{\z_{t+1} \z_{t} } | = \frac{1}{L_0} \| \grad f (\z_t) \| \le \frac{1}{\mu} \max_{\z \in \D_0} \|  \grad f(\z) \|$. 
    Hence $\z_{t+1}$ belong to $\D_0$. We hope to proof $\z_{t+1} \in \D$ by establishing $f (\z_{t+1}) \le f (\z_{t})$.
    It suffice to show that
    \begin{align*}
        f (\z_{t+1}) - f(\z^*)
        \le
        \( 1-\frac{\mu}{4 L_0} \) \( f (\z_t) - f(\z^*) \),
    \end{align*}
    which will be proved next.

    As $L_0$-smoothness holds between $\z_t$ and $\z_{t+1}$, we have 
    \begin{align}
        f (\z_{t+1}) 
        \le &
        f (\z_t) + \< \grad f (\z_t) , \Exp_{\z_t}^{-1} \(\z_{t+1}\) \>_{\z_t} + \frac{L_0}{2} | \oa{\z_{t+1} \z_{t} } |^2 \nonumber \\ 
        =& \; 
        f (\z_t) + \( - \frac{1}{L_0} + \frac{1}{2 L_0} \) \| \grad f (\z_t) \|_{\z_t}^2 \nonumber \\
        = & \;
        f (\z_t) - \frac{1}{2 L_0} \| \grad f (\z_t) \|_{\z_t}^2. 
        \label{eq:tmp3}
    \end{align}

    By property of strong convexity, we have
    \begin{align}
        f(\z_{t}) - f(\z^*) \le \frac{2}{\mu} \| \grad f(\z_t) \|_{\z_t}^2.
        \label{eq:tmp4}
    \end{align}

    Plugging Eq. (\ref{eq:tmp4}) into Eq. (\ref{eq:tmp3}) gives 
    \begin{align*} 
        f (\z_{t+1}) 
        \le
        \( 1-\frac{\mu}{4 L_0} \) f (\z_t) + \frac{\mu}{ 4 L_0} f(\z^*). 
    \end{align*} 
    
    Now we rearrange terms to get 
    \begin{align*} 
        f (\z_{t+1}) - f(\z^*)
        \le
        \( 1-\frac{\mu}{4 L_0} \) \( f (\z_t) - f(\z^*) \).
    \end{align*}

\end{proof}

\begin{remark} 
    % When solving a specific proximal step in the algorithm \eqref{eq:proximal-iter}, we can set $\z_0 = \x_t$, and the $\x_{t+1} = \z^*$ will be close to $\z_0$. 
    When solving the proximal step in \eqref{eq:proximal-iter}, we can initialize Algorithm \ref{alg:gd} with $\z_0 = \x_t$. Since the solution $\x_{t+1} = \z^*$ will be close to $\z_0$, this provides a warm start. 
\end{remark} 

%---------------------

%% file: tex/appendix.tex
\section{Some Known Facts}

% \section{Omitted Proofs} 

In the Appendix we list some known facts, with proofs provided for completeness. 

\begin{proposition}
    \label{prop:basic}
    Let $f$ be convex and $L$-smooth with $\x^*$ as its minimum. Then for any $\x$ it holds that 
    \begin{align*}
        f (\x^* ) 
        \le 
        f (\x) - \frac{1}{2L} \| \grad f (\x) \|_{\x}^2 , \quad \forall \x . 
    \end{align*}
\end{proposition}

\begin{proof}
% [Proof of Proposition \ref{prop:basic}]
    It holds that 
    \begin{align*} 
        f (\x^* ) 
        \le& \;  
        f \( \Exp_{\x} \( - \frac{1}{L} \grad f (\x) \) \) \\ 
        \le& \;  
        f (\x) + \< \grad f (\x) , - \frac{1}{L} \grad f (\x) \>_{\x} + \frac{L}{2} \cdot \frac{1}{L^2} \| \grad f (\x) \|_{\x}^2 
        = 
        f (\x) - \frac{1}{2L} \| \grad f (\x) \|_{\x}^2 , 
        % \< \grad f (\x) , - \frac{1}{L} \grad f (\x) \>_{\x_t} + \frac{L}{2} \cdot \frac{1}{L^2} \| \grad f (\x_t) \|_{\x_t}^2 
    \end{align*} 
    where the former inequality comes from minimum $\x^*$ and the latter one uses definition of $L$-smooth.
\end{proof}

\begin{proposition}
    \label{prop:basic-2}
    Let $f$ be $\mu$-strongly $g$-convex with $\x^*$ as its minimum. Then for any $\x$ it holds that 
    \begin{align*}
        f (\x ) - f (\x^*)
        \le 
         \frac{2}{\mu} \| \grad f (\x) \|_{\x}^2 , \quad \forall \x . 
    \end{align*}
\end{proposition}

\begin{proof}
% [Proof of Proposition \ref{prop:basic-2}]
    By strongly $g$-convex,
    \begin{align}
        f(\x^*) \ge f(\x) + \< \grad f(\x), \Exp^{-1}_{\x}(\x^*) \>_{\x} + \frac{\mu}{2} d(\x,\x^*)^2.
        \label{eq:prop-basic-2}
    \end{align}
    Because $\x^*$ is minimum, $f(\x^*)-f(\x)\le 0$. Combined with Cauchy--Schwarz inequality,
    \begin{align*}
        \frac{\mu}{2} d(\x,\x^*)^2 
        \le
        - \< \grad f(\x), \Exp^{-1}_{\x}(\x^*) \>_{\x}
        \le 
        \| \grad f(\x) \|_{\x} d(\x,\x^*).
    \end{align*}
    Hence, $\| \grad f(\x) \|_{\x} \ge \frac{\mu}{2}d(\x,\x^*)$. Inserting this result into Eq. (\ref{eq:prop-basic-2}) gives
    \begin{align*}
        f(\x) - f(\x^*) 
        \le 
        \| \grad f(\x)\|_{\x} d(\x,\x^*) 
        \le 
        \frac{2}{\mu} \| \grad f(\x)\|_{\x}^2.
    \end{align*}
\end{proof}